\documentclass[a4paper,11pt,reqno]{amsart}
\usepackage{amssymb,mathrsfs}
\usepackage{amsthm}
\usepackage{amsmath}
\usepackage{latexsym}
\usepackage{fancyhdr}
\usepackage{ascmac}
\usepackage{color}
\usepackage[all]{xy}
\usepackage{amscd}
\usepackage{listings}
\usepackage{xcolor}
\usepackage{caption}

\theoremstyle{plain}
\newtheorem{thm}{Theorem}[section]
\newtheorem{prop}[thm]{Proposition}

\newtheorem{lem}[thm]{Lemma}

\makeatletter

\@addtoreset{equation}{section}
\makeatother

\newcommand{\bQ}{\overline{\mathbb{Q}}}

\newcommand{\bZ}{\overline{\mathbb{Z}}}

\newcommand{\bF}{\overline{\mathbb{F}}}

\newcommand{\C}{\mathbb{C}}
\newcommand{\R}{\mathbb{R}}
\newcommand{\Q}{\mathbb{Q}}

\newcommand{\Z}{\mathbb{Z}}

\newcommand{\F}{\mathbb{F}}

\newcommand{\lra}{\longrightarrow}

\newcommand{\HH}{\mathbb{H}}
\newcommand{\vp}{\varphi}

\renewcommand{\O}{\mathcal{O}}

\newcommand{\Aut}{{\rm Aut}}
\newcommand{\Bir}{{\rm Bir}}

\newcommand{\ds}{\displaystyle}

\newcommand{\G}{\Gamma}

\newcommand{\lcm}{{\rm lcm}}
\newcommand{\la}{\lambda}

\newcommand{\SL}{{\rm SL}}

\newcommand{\bs}{\backslash}

\newcommand{\ord}{{\rm ord}}

\newcommand{\wW}{\widetilde{W}}

\newcommand{\ve}{\varepsilon}
\newcommand{\PP}{\mathbb{P}}

\title[An Explicit Belyi Map for the Wiman Sextic and...]
{An Explicit Belyi Map for the Wiman Sextic and Cusp Forms for a Noncongruence Subgroup}
\author{Madoka Horie and Takuya Yamauchi}
\keywords{the Wiman curve, Belyi maps, noncongruence modular forms, the unbounded denominator conjecture for weight 2 cusp forms.}
\thanks{}
\subjclass[2020]{14G10, 11G35}

\address{Madoka Horie \\
Faculty of Fundamental Science, 
National Institute of Technology, 
Niihama College., 102-8554, JAPAN}
\email{horiemaaa@gmail.com}

\address{Takuya Yamauchi \\ 
Mathematical Inst. Tohoku Univ.\\
 6-3,Aoba, Aramaki, Aoba-Ku, Sendai 980-8578, JAPAN}
\email{takuya.yamauchi.c3@tohoku.ac.jp}

\begin{document}
\maketitle

\begin{abstract}
In this paper, we explicitly determine an algebraic Belyi function on a unique smooth projective model $\wW$ of the Wiman sextic curve $W$ and describe its complex uniformization in terms of modular functions associated with a certain noncongruence subgroup $\G_{\wW} \subset \SL_2(\Z)$. As an application, we give a direct proof of the unbounded denominators conjecture for weight \(2\) cusp forms with respect to $\G_{\wW}$. The conjecture is now known in full generality by the work of Calegari, Dimitrov, and Tang, following earlier progress including work of Dong, Lin, and Ng. Our proof, however, uses a degeneration of $\wW$ over $\F_{5}$ together with explicit Puiseux series expansions and is substantially different from the methods employed in their work.
\end{abstract}

\tableofcontents

\section{Introduction}\label{intro}

Belyi's fundamental theorem \cite{Belyi} characterizes algebraic curves over
number fields in terms of their ramification. More precisely, a complex
projective smooth curve $C$ admits a model over a number field if and only if
there is a non-constant morphism
\[
  \beta:C\lra \mathbb{P}^1
\]
whose branch locus is contained in $\{0,1,\infty\}$. We shall refer to such a
morphism as an algebraic Belyi function. When both $C$ and $\beta$ are defined
over a number field $K$, we say that $\beta$ is an algebraic Belyi function
defined over $K$.

Consequently, every projective smooth geometrically connected curve over a
number field possesses an algebraic Belyi function, possibly after replacing
the ground field by a finite extension. The theorem, however, is essentially
an existence result: the problem of constructing such a function explicitly
for a given curve is generally non-trivial. A substantial body of work has
therefore been devoted to the explicit determination of algebraic Belyi
functions; see, for example, \cite{SS} and \cite{SV}.

From the viewpoint of hyperbolic uniformization, a Belyi function may be
realized as a natural quotient map. There are two relevant realizations,
according to whether the uniformizing group is non-cocompact or cocompact. 
Let $\mathbb{H}:=\{\tau\in\C\mid {\rm Im}(\tau)>0\}$. 
\begin{enumerate}
\item[(1)] \textit{Non-compact uniformization.}
There exist a subgroup $\G\subset  \SL_2(\Z)$ of finite index and 
a biholomorphism 
\[
  X(\G):=\G\bs
  \bigl(\mathbb{H}\cup\mathbb{P}^1(\Q)\bigr)
  \stackrel{\sim}{\lra} C(\C)
 \]
such that 
$\beta$ is interpreted as the natural projection $$X(\G)\lra X(\SL_2(\Z))\stackrel{\sim}{\lra} \mathbb{P}^1(\C)
,\ \G\tau\mapsto \SL_2(\Z)\tau\mapsto \frac{1}{1728}j(\tau)$$  
where $j$ is Klein's $j$-invariant function, with expansion $j(\tau)=q^{-1}+744+196884q+\cdots,\ 
q:=e^{2\pi \sqrt{-1}\tau},\ \tau\in\mathbb{H}$. 

We remark that if $\G\subset \G(2)$, then 
$X(\G)\lra X(2)\stackrel{\la\atop \sim}{\lra}\PP^1(\C)$ is also a Belyi function 
where $\la$ is the modular lambda function (see (\ref{lambda}) or \cite[VII, Section 7]{Chan}). 

\item[(2)] \textit{Compact uniformization.}
There exist a triangle group
$\Delta(a,b,c)\subset\SL_2(\R)$, a subgroup
$\Delta\subset\Delta(a,b,c)$ of finite index, and a biholomorphism
\[
  X_\Delta:=\Delta\bs\mathbb{H}
  \stackrel{\sim}{\lra}C(\C)
\]
such that $\beta$ corresponds to the quotient map
\[
  X_\Delta\lra X_{\Delta(a,b,c)}
  \stackrel{\sim}{\lra}\mathbb{P}^1(\C),
  \quad
  \Delta\tau\longmapsto\Delta(a,b,c)\tau
  \longmapsto j_{\Delta(a,b,c)}(\tau),
\]
where $j_{\Delta(a,b,c)}$ is a uniformizing function on
$X_{\Delta(a,b,c)}$. 
\end{enumerate}
For these two uniformization statements, we refer respectively to
\cite[p.~71, Theorem]{Se97} and
\cite[Section~3.3, Theorem~3]{Wolfart}.

In this paper, we first compute an algebraic Belyi function on the Wiman curve 
over $\Q$ defined by
\begin{equation}\label{Bring}
W:
f(X,Y,Z):=X^6+Y^6+Z^6+(X^2+Y^2+Z^2)(X^4+Y^4+Z^4)-12 X^2 Y^2 Z^2=0 
\end{equation}
inside $\mathbb{P}^2$ with homogeneous coordinates $[X:Y:Z]$. 
This curve has four ordinary double points $[\pm 1:\pm 1:1]$ as its singularities. 
By the Max Noether's formula (\cite[p.614, Theorem 5]{BK}), the geometric genus $g(W)$ of $W$ 
can be calculated as 
$$g(W)=\frac{(6-1)(6-2)}{2}-4\times \frac{2(2-1)}{2}=6.$$
The full birational automorphism group $\Bir(W)$ is isomorphic to the symmetric group $S_5$ of five letters. 
Explicitly, it is generated by $[X:Y:Z]\mapsto [\ve_1 X:\ve_2 Y:\ve_3 Z]$ 
($(\ve_1,\ve_2,\ve_3)\in \mu_2^3/\Delta \mu_2$), the permutations of coordinates 
which makes up $\mu_2^3/\Delta \mu_2\rtimes S_3\simeq S_4$, and plus 
the automorphism $\alpha:[X:Y:Z]\mapsto [g_0:g_1:g_2]$ of order five where 
\[
\begin{array}{l}
g_0=-X^2 + X Y - Y^2 + X Z - Y Z + Z^2,\\
g_1= -X^2 + X Y + Y^2 + X Z - Y Z - Z^2,\\
 g_2=X^2 + X Y - Y^2 + X Z - Y Z - Z^2.
 \end{array}
 \]
More precisely, $[X:Y:Z]\mapsto [-X:Y:Z]$,  
$[X:Y:Z]\mapsto [X:-Y:Z]$, and $\alpha$  corresponds to $(12)(34),\ (13)(24)$ and 
$(14325)$ respectively. 
We refer to \cite[Section 3]{GKLM} for the birational automorphism group of $W$. 
Let $\wW$ be a unique smooth model of $W$ over $\Q$ which is obtained by 
blowing up four ordinary double points (see Proposition \ref{diff-can}(2) for 
another projective model). 
Obviously, $\Aut_\Q(\wW)=\Aut_{\bQ}(\wW)\simeq \Bir_{\bQ}(W)=\Bir_{\Q}(W)\simeq S_5$ and $g(\wW)=g(W)=6$. 

\begin{thm}\label{main1}{\rm(}Theorem \ref{mainclaim} {\rm)} The rational function 
$$\beta(X,Y,Z)=\frac{3125 \left(\frac{3 \left(X^2 Y^2+Y^2 Z^2+Z^2 X^2\right)}{\left(X^2+Y^2+Z^2\right)^2}-1\right)^2 \left(\frac{X^2 Y^2+Y^2 Z^2+Z^2 X^2}{\left(X^2+Y^2+Z^2\right)^2}+5\right)^3}{\left(\frac{5 \left(X^2 Y^2+Y^2 Z^2+Z^2 X^2\right)}{\left(X^2+Y^2+Z^2\right)^2}-2\right) \left(-\frac{75 \left(X^2 Y^2+Y^2 Z^2+Z^2 X^2\right)^2}{\left(X^2+Y^2+Z^2\right)^4}+\frac{1010 \left(X^2 Y^2+Y^2 Z^2+Z^2 X^2\right)}{\left(X^2+Y^2+Z^2\right)^2}+13\right)^2}$$
on $\wW$ 
is a generator of $\Q(\wW)^{S_5}=\Q(W)^{S_5}$ and it gives rise to an algebraic Belyi map over $\Q$:
$$\beta:\wW\lra S_5\bs \wW\simeq \PP^1,\ [X:Y:Z]\mapsto \beta(X,Y,Z).$$
\end{thm}
We identify $\Aut_\Q(\wW)=\Aut_{\bQ}(\wW)=\Aut(\beta)\simeq {\rm Gal}(\C(\wW)/\beta^\ast(\C(\PP^1)))$ with $S_5$. 
For each $x\in \{0,1,\infty\}\subset \PP^1$ and $y\in \beta^{-1}(x)$, we denote by $I_x(y)=
\{g\in S_5\ |\ gy=y\}$ the inertia group at $y$.  
By \cite[Lemma 2.9]{FL} (see also \cite{EKS} for terminology), 
$$I_0(y)\sim \langle \sigma_0:=(132)(45)\rangle,\  
I_1(y)\sim \langle \sigma_1:=(1542)\rangle,\ 
I_\infty(y)\sim \langle \sigma_\infty:=(14)(23)\rangle$$
where $\sim$ means a conjugacy relation in $S_5$. We remark that \cite[Lemma 2.9]{FL} uses a right action, whereas we use a left action. Consequently, their elements $\sigma_0,\sigma_1,\sigma_\infty$ are the inverses of ours. 
Using a uniformization 
$$\G(2)\bs \mathbb{H}\stackrel{\sim}{\lra} \PP^1\setminus\{0,1,\infty\},\ \tau\mapsto 
\lambda(\tau)=\frac{\theta_2(\tau)^4}{\theta_3(\tau)^4},$$
where $\theta_2(\tau)=\ds\sum_{n\in \Z}q_2^{(n+\frac{1}{2})^2},\ 
\theta_3(\tau)=\ds\sum_{n\in \Z}q_2^{n^2},\ q_2:=e^{\pi i \tau }$ 
 (see \cite[VII, Section 7]{Chan}), 
we can define a finite index subgroup $\G_{\wW} \subset \G(2)\subset \SL_2(\Z)$ making the 
following diagram commutative:
 \[
\xymatrix{
\G(2)\bs\mathbb{H} \ar[r]^{\lambda \hspace{5mm} \atop \sim \hspace{5mm}} & \PP^1\setminus\{0,1,\infty\} \\
\G_{\wW}\bs\mathbb{H}\ar[u]^{\text{the natural} \atop \text{projection}}  
\ar[r]^{\sim \hspace{12mm}} & \wW\setminus \beta^{-1}(\{0,1,\infty\}) \ar[u]_\beta.
}
\]
Explicitly, if we fix generators $T^2,\ ST^2 S^{-1}$ of $\G(2)$ with 
$
T=\begin{pmatrix}
1 & 1 \\
0 & 1
\end{pmatrix}
$
and 
$
S=\begin{pmatrix}
0 & 1 \\
-1 & 0
\end{pmatrix}
$, 
then $\G_{\wW}$ is given  as the kernel of the surjective homomorphism 
$$\rho:\G(2)\lra S_5,\ 
\begin{array}{c}
T^2\mapsto \sigma_0 \\
ST^2S^{-1}\mapsto \sigma_1 
\end{array}.
 $$
Thus, $\G_{\wW}$ is of index $[\SL_2(\Z):\G(2)]\times |S_5|=720$. As we will see later,  
$\G_{\wW}$ is a non-congruence subgroup and the width of the cusp $\infty$ of $X_{\wW}:=X(\G_{\wW})$  is 12 (see Lemma \ref{noncong}). Let $q_{12}=e^{2\pi i \tau/12}$ be 
a local coordinate of $X_{\wW}$ at $\infty$. 
By using an elementary method which is quite different from the methods in previous results 
\cite{CDT}, \cite{DLN} (see also \cite{Fresan}), we prove the unbounded denominator 
conjecture for the space $S_2(\G_{\wW})$ consisting of 
all cusp forms of weight 2 with respect to $\G_{\wW}$:
\begin{thm}\label{main2} $($Theorem \ref{udc}$)$ Any cusp form in $S_2(\G_{\wW})$ all of whose Fourier coefficients of its $q_{12}$-expansion belong to $\bQ$ has unbounded denominators. 
In particular, the unbounded denominator conjecture for weight 2 cusp forms is true for $\G_{\wW}$. 
\end{thm}
We should note that the unbounded denominators conjecture 
for any weak modular forms has already been established for all noncongruence subgroups of $\SL_2(\Z)$ by Dong, Lin, and Ng \cite{DLN} and by Calegari, Dimitrov, and Tang \cite{CDT}. Our approach, however, is based on a degeneration of the Wiman curve over the finite field $\F_{5}$ together with explicit Puiseux series expansions, and is quite different from their methods. It would be of independent interest to extend our approach to more general noncongruence subgroups of $\SL_2(\Z)$.

The paper is organized as follows. In Section \ref{bwc}, we recall some basic facts about the Wiman curve over an arbitrary field of characteristic different from $2,3,5$ and study certain properties which 
are not previously covered.  
Section \ref{abfWiman} is devoted to computing an explicit Belyi function on $\wW$ over $\Q$ and applying Galois theory to the covering induced by 
the full automorphism group of $\wW$. In Section \ref{modularparacuspform}, we describe a complex uniformization of the Wiman curve and analyze the Fourier expansions of cusp forms for $\G_{\wW}$. We present two approaches to the uniformization: the first is a direct application of Puiseux series expansions, while the second is better suited to the study of cusp forms 
of weight 2. Finally, in Section \ref{unbounded-dc}, we prove the unbounded denominators conjecture 
in weight 2 cusp forms for $\G_{\wW}$. 
The appendix is devoted to computing an explicit generator of the invariant field 
$\Q(\wW)^{A_5}$ associated with the $A_5$-covering $\wW\lra A_5\bs \wW\simeq \PP^1$.

\textbf{Acknowledgments.} We would like to thank Professor Ling Long 
for valuable comments. 
We also thank Professor Hiroshi Tsunogai for his helpful comments and for raising the question concerning the $A_5$-covering discussed in the appendix.
The second author was partially supported supported by JSPS KAKENHI Grant Number 
JP26K22241.

\subsection{Notation}
We conclude this introduction with some notation used throughout the paper. 
\begin{itemize}
\item For a field $k$, let $k^{{\rm sep }}$ be the separable closure of $k$ in 
a fixed algebraically closed field of $k$.    
\item Let $\bQ$ (resp. $\bZ$) be the set of all algebraic numbers (resp. algebraic integers) in $\C$.
\item For an algebraic curve $X$ over a field $k$ and an extension $L/k$, 
we denote by $\Aut_L(X)$ (resp. $\Bir_L(X)$) the group of all automorphisms 
(all birational morphisms) defined over $L$. 
For a Galois covering $\beta:X\lra Y$ over a field $L$ between algebraic curves over $L$, 
we denote by $\Aut_L(\beta)$ the group of all automorphisms over $L$ 
commuting with $\beta$. We drop $L$ if $L$ is an algebraically closed field. 
\item For each rational prime $p$, let $\Q_p$ (resp. $\Z_p$) be the field of $p$-adic numbers 
($p$-adic integers). 
\item If $K$ is an algebraic extension of $\Q$ or $\Q_p$ with $p$ prime, 
we denote by $\O_K$ the ring of integers of $K$.  
\item For a set $X$, we denote by $|X|$ its cardinality, 
\end{itemize}

\section{Basic properties of the Wiman curve over a field}\label{bwc}
For the contents in this section, we refer to \cite{GKLM}. 
Let $k$ be a field of characteristic different from $2,3,5$.  
The Wiman curve $W=W_k$ over $k$ is defined by the following sextic equation 
\begin{equation}
  W:\quad
  f(X,Y,Z):=X^6+Y^6+Z^6+(X^2+Y^2+Z^2)(X^4+Y^4+Z^4)  -12X^2 Y^2 Z^2=0
\end{equation}
inside $\PP^2$ with the projective coordinates $[X:Y:Z]$. 
The curve $W$ has four ordinary double points $[\pm 1:\pm 1:1]$ as the singularity, and its normalization $\wW=\wW_k$ is  a smooth 
projective geometrically connected algebraic curve over $k$ of genus $6$ (\cite{GKLM}). We also call $\wW$ the Wiman curve over $k$. It is easy to see that the blowing up $\pi:\wW\lra W$ at the 
center that is the subvariety $Z$ of the four ordinary double points is 
defined over $k$. The strict transform of each point of $Z$ consists of two points exactly defined over $k(\sqrt{-3})$ but as a scheme it is defined over $k$. 
Thus, $|\widetilde{Z}(k)|=8$ if $k$ contains $\sqrt{-3}$, $|\widetilde{Z}(k)|=0$ otherwise where 
$\widetilde{Z}$ is the strict transform of $Z$ regarded as a $k$-scheme.

The full automorphism group $\Aut_{k^{{\rm sep }}}(\wW)$ is isomorphic to $S_5$ (\cite[Theorem 3.4]{GKLM}) and its generators are those explained in Section \ref{intro}. 
Since the singular points and the generators of $\Aut_{k^{{\rm sep }}}(\wW)$ explained as above are defined over the base field, $\Aut_{k}(\wW)$ coincides with $\Aut_{k^{{\rm sep }}}(\wW)$. 

\begin{prop}\label{basicpro}Keep the assumption on $k$. 
For the Wiman curve $\wW$ over $k$, it holds the following properties:
\begin{enumerate}
\item $\wW$ is non-hyperelliptic.
\item $\wW$ is not trigonal. Namely, $\wW$ does not admit any maps $\wW\lra \PP^1$ over $k^{{\rm sep }}$ of degree 3. 
\item The Jacobian variety $J(\wW)$ of $\wW$ is isogenous over $k$ to 
the 6-th product $E^6$ of which $E$ is the elliptic curve defined by $y^2+xy=x^3-828x+9072$. 
If $k=\Q$, the elliptic curve $E$ is labelled as LMFDB label 150.c1.
\item if $k=\Q$, $\wW$ has good reduction at each prime $p$ other than $2,3,5$. Namely, $\wW_{\Q_p}$ 
has an integral smooth model $\widetilde{\mathcal{W}}$ over $\Z_p$ such that
$\widetilde{\mathcal{W}}\otimes_{\Z_p}\Q_p\simeq \wW_{\Q_p}$ and 
$\widetilde{\mathcal{W}}\otimes_{\Z_p}\F_p\simeq \wW_{\F_p}$. 
\end{enumerate}
\end{prop}
\begin{proof}For the fourth claim, by direct computation, we can find 
a smooth model over $\Z[1/30]$ by using blowing-ups (\cite[Section 8.1]{Liu}). 
The third claim follows from \cite[Lemma 2.2]{Kawa} and the label of $E$ when $k=\Q$ 
can be read off in \cite{LMFDB}. 

For the first claim, the quotient of $\wW$ by the involution $[X:Y:Z]\mapsto [Y:X:Z]$ yields 
a quartic plane curve in $\PP^2_{[u:v:w]}$ defined by 
$$19 u^2 v^2-12 u^3 v+13 u^2 w^2+2 u^4-6 u v^3-42 u v w^2+8 v^2 w^2+24 w^4=0$$
with $(u,v,w)=((X+Y)^2,X Y+Z^2,Z (X+Y))$. By direct computation, it is smooth. 
Thus, it is a non-hyperelliptic (cf. \cite[IV, p.315, EXERCISES 3.2(c)]{Ha}). Any quotient of a hyperelliptic curve can not be non-hyperelliptic (notice that the hyperelliptic involution 
is in the center of the full automorphism group). 
The claim follows. 

Finally, we prove the second claim. 
Assume $\wW_{\bQ}$ is trigonal. Hence, there exists a morphism 
$\wW_{\bQ}\lra \PP^1_{\bQ}$ of degree 3. By \cite[Theorem 2.1]{NS}, 
there exists a morphism $\wW_{\Q}\lra \PP^1_{\Q}$ over $\Q$ of degree 3. 
Let $\ell=13$. Apply \cite[Lemma 5.1]{NS} and the fourth claim, there exists  
 a morphism $\wW_{\F_\ell}\lra \PP^1_{\F_\ell}$ over $\F_\ell$ of degree less than or equal to 3. 
 Thus, we have $|\wW_{\F_\ell}(\F_\ell)|\le 3 |\PP^1_{\F_\ell}(\F_\ell)|=42$. 
 However, $|\wW_{\F_\ell}(\F_\ell)|=|W_{\F_\ell}(\F_\ell)|+|\widetilde{Z}(\F_\ell)|-
 |Z(\F_\ell)|=46+8-4=50$ 
 since $\sqrt{-3}\in \F_{\ell}$. Thus, it gives a contradiction. 
 Therefore, $\wW_{\bQ}$ is not trigonal and so is $\wW_\C$.  
 
We now turn to the general case. Let $k'$ be the prime field contained in $k$. 
 If $k'=\Q$, it is done. Thus, we assume $k'=\F_p$ for some prime $p$ different from $2,3,5$.  
 Assume $\wW_{\overline{k}}$ is trigonal. 
 Then, it is easy to see that $\wW_{k_1}$ is trigonal over $k_1$ for some finite extension 
 $k_1/k'$. 
 We briefly explain this. It is well-known that the functor ${\rm Mor}^{(3)}_{k'}(\wW_{k'},\PP^1_{k'})$ consisting of all degree 3 morphisms from $\wW_{k'}$ to $\PP^1_{k'}$ 
 is a finite type $k'$-scheme. In fact, the functor ${\rm Mor}_{k'}(\wW_{k'},\PP^1_{k'})$ is 
 represented by an open subscheme of the Hilbert scheme ${\rm Hilb}_{(\wW\times_{k'} \PP^1)/k'}$. Thus, ${\rm Mor}^{(3)}_{k'}(\wW_{k'},\PP^1_{k'})$ is an open subscheme of a finite type 
 scheme ${\rm Hilb}^P_{(\wW\times_{k'} \PP^1)/k'}$ for the Hilbert polynomial $P$ 
 given by $P(n)=(\deg(K_{\wW})+3)n+1-g(\wW)=13n-5$   
 (cf. \cite[Theorem 5.20,Theorem 5.23]{Fantechi}). Hence, ${\rm Mor}^{(3)}_{k'}(\wW_{k'},\PP^1_{k'})$ 
 is a finite type $k'$-scheme. 
 It is non-empty since it has a $k$-rational point by assumption. 
Therefore, ${\rm Mor}^{(3)}_{k'}(\wW_{k'},\PP^1_{k'})$ has a $k_1$-rational point 
for a finite extension $k_1/k'$ by Zariski's lemma. 
 Thus, we may assume that $\wW_{k_1}$ is trigonal so that its base change 
 $\wW_{\overline{k}'}$ is also trigonal.  
 By \cite[Theorem 2.1]{NS} again, 
there exists a morphism $f:\wW_{k'}\lra \PP^1_{k'}$ over $k'$ of degree 3. 
Since $p\neq 3=\deg(f)$, 
the ramification of $f$ is tame. By \cite[Proposition 4.5.]{RW}, 
there exists a finite extension $K/\Q_p$ such that 
$\tilde{f}:\wW_{\O_K}\lra \PP^1_{\O_K}$ is a lift of the base change of $f$ to the residue field of $K$. Then, its generic fiber yields a morphism $\wW_K\lra \PP^1_K$ of degree 3 since 
it is flat over $\O_K$.  
However, it contradicts with the claim for the characteristic zero case under an identification 
$\overline{K}\simeq \C$. 
\end{proof}

\begin{prop}\label{diff-can} Keep the notation as above. 
Let $K_{\wW}$ denote a canonical divisor of $\wW$. Let $x=X/Z$ and $y=Y/Z$. 
\begin{enumerate}
\item $H^0(\wW,\Omega^1_{\wW})$ has 
a basis given by 
$$
  \omega_1
  =
  \frac{x(x^2-1)\,dx}{F_y(x,y)},\ 
  \omega_2
  =
  \frac{y(x^2-1)\,dx}{F_y(x,y)},\ 
  \omega_3
  =
  \frac{(x^2-1)\,dx}{F_y(x,y)},
$$
$$
  \omega_4
  =
  \frac{x(y^2-1)\,dx}{F_y(x,y)},\ 
  \omega_5
  =
  \frac{y(y^2-1)\,dx}{F_y(x,y)},\ 
  \omega_6
  =
  \frac{(y^2-1)\,dx}{F_y(x,y)}
$$
where $F_y(x,y)=\ds\frac{\partial f(x,y,1)}{\partial y}$ is the formal partial derivative 
of $f(x,y,1)$ in $y$.  
\item The canonical map $\Phi_{|K_{\wW}|}:\wW_k\lra \PP^5_k$ sending $P$ to 
$[\omega_1(P):\omega_2(P):\omega_3(P):\omega_4(P):\omega_5(P):\omega_6(P)]$ is 
a closed embedding over $k$ and its image is a smooth projective geometrically connected model over $k$ of $\wW$ in $\PP^5_k$  whose defining equations are given by 
\[
\begin{aligned}
2x_1^2+x_2^2+5x_3^2+x_4^2+2x_5^2-8x_3x_6+5x_6^2&=0,\\
x_2^2-x_3^2-x_1x_4+x_3x_6&=0,\\
x_2x_4-x_1x_5&=0,\\
x_3x_4-x_1x_6&=0,\\
x_4^2-x_2x_5+x_3x_6-x_6^2&=0,\\
x_3x_5-x_2x_6&=0,
\end{aligned}
\]
where $(x_1,x_2,x_3,x_4,x_5,x_6)=(x(x^2-1),y(x^2-1),x^2-1,x(y^2-1),y(y^2-1),y^2-1)$. 
\end{enumerate}
\end{prop}
\begin{proof}The first claim follows by direct computation. 
For the second claim, it  follows from Petri's theorem (see \cite[Section 2.2]{Shi}) with 
the first claim and Proposition \ref{basicpro}(2) that the image is defined by 
$\frac{1}{2}(g(\wW)-2)(g(\wW)-3)=6$ quadratic equations. Here we discarded any cubic relations by Proposition \ref{basicpro}(2). 
  We may determine these equations by the method of undetermined coefficients using the relation $f(x,y,1)=0$. 
\end{proof}

\section{An algebraic Belyi function on the Wiman curve over $\Q$}\label{abfWiman} 
In this section, we explicitly construct a Belyi function on the Wiman curve $\wW$ over $\C$ by using 
the quotient by the full automorphism group and its non-Galois covering of degree 5 as an intermediate step. 

We write $G:=\Aut_{\C}(\wW)$ and identify it with  $S_5$ 
and let $G$ act on $\wW$ from the left. 
As explained in Section \ref{intro}, by \cite[Theorem 2.1, Lemma 2.9]{FL}, 
the quotient $G\bs \wW$ is isomorphic to $\mathbb{P}^1_\lambda$ with a parameter $\lambda$, and the corresponding Galois cover
\begin{equation}\label{GC}
  \pi_G:\wW \longrightarrow G\bs\wW \simeq \mathbb{P}^1_\lambda 
\end{equation}
has three branch points normalized to be $0,1,\infty$. 
Then, for each $y_x\in \pi_G^{-1}(x)$ with $x\in \{0,1,\infty\}$, 
\begin{equation}
  I_0(y_0)\sim \langle \sigma_0=(132)(45)\rangle,\ 
  I_1(y_1)\sim \langle \sigma_1=(1542)\rangle,\ 
  I_\infty(y_\infty)\sim \langle \sigma_\infty=(14)(23)\rangle
\end{equation}
and they 
have orders $6,4,2$, respectively. 
We fix such $y_0,y_1,y_\infty$ and simply write $I_x=  I_x(y_x)$ for each $x\in\{0,1,\infty\}$. 

\subsection{A certain intermediate quotient}
Let $H$ be the subgroup of $G$ which is isomorphic to $S_4=\{\sigma\in S_5\ |\ \sigma(5)=5\}$. 
Put $Y=H\bs \wW$. Let $\pi_{\wW,Y}:\wW\lra Y$ be the natural projection.  
The natural quotient map induced from the inclusion $H\subset G$ 
\begin{equation}
\alpha: Y\lra G\bs \wW=\mathbb{P}^1
\end{equation}
 is a non-Galois cover of degree $5$. 
For each $x\in \{0,1,\infty\}$, 
we see easily that $\alpha^{-1}(x)=\pi_{\wW,Y}(\pi^{-1}_G(x))=H\bs Gx$ 
is identified with 
$$H\bs G/I_x.$$
The quotient  $H\backslash G$ can be identified
with $\{1,\dots,5\}$ by
\begin{equation}
  Hg\longmapsto g^{-1}(5).
\end{equation}
Under this identification, the right multiplication by $\sigma\in G$ on $H\bs G$ corresponds to
the usual left action of $\sigma^{-1}$ on $\{1,\dots,5\}$. 
Put $C_i:=H(i5)$ for $1\le i\le 5$. Then, $H\bs G=\{C_{1},\ldots,C_5\}$ and the above identification sends $C_i$ to $i$ for each 
$1\le i\le 5$. 
We now compute the $I_x$-orbits of $H\bs G=\{C_{1},\ldots,C_5\}$. 
For $x=0$, 
$$C_1\stackrel{\sigma_0}{\lra}C_2\stackrel{\sigma_0}{\lra}C_3\stackrel{\sigma_0}{\lra} C_1,\ 
C_4\stackrel{\sigma_0}{\lra}C_5\stackrel{\sigma_0}{\lra}C_4.$$
Thus, $H\bs G=\{\O_{I_0}(C_1),\ \O_{I_0}(C_4)\}$ with $|\O_{I_0}(C_1)|=3$ and $|\O_{I_0}(C_4)|=2$. 
Therefore, $\alpha^{-1}(0)$ consists of two points of ramification indices $3$ and $2$ respectively. 
Similarly, we have $H\bs G=\{\O_{I_1}(C_1),\ \O_{I_1}(C_5)\}$ with $|\O_{I_1}(C_1)|=4$ and $|\O_{I_1}(C_5)|=1$. 
Therefore, $\alpha^{-1}(1)$ consists of two points of ramification indices $4$ and $1$ respectively. 
Finally, we have $H\bs G=\{\O_{I_\infty}(C_1),\ \O_{I_\infty}(C_2),\ \O_{I_\infty}(C_5)\}$ with $|\O_{I_\infty}(C_1)|=
|\O_{I_\infty}(C_2)|=2$ and $|\O_{I_\infty}(C_5)|=1$. 
Therefore, $\alpha^{-1}(\infty)$ consists of three points of ramification indices $2,2$ and $1$ respectively. 
By Riemann-Hurwitz formula, 
$$2g(Y)-2=5(-2)+(3-1)+(2-1)+(4-1)+(1-1)+(2-1)+(2-1)+(1-1)=-2$$
and thus $g(Y)=0$. 

\begin{prop}\label{BelyiY} Keep the notation as above. 
There exists an isomorphism $\mathbb{P}^1_T\simeq Y$ such that 
the composition $\vp:\mathbb{P}^1_T\simeq Y\stackrel{\alpha}{\lra} \PP^1_\la$ is an algebraic Belyi function over $\Q$ given by 
\begin{equation}\label{laT} \lambda= -\frac{50000T^3(T-1)^2}
  {(375T^2-50T-1)^2}.
  \end{equation}
\end{prop}
\begin{proof}Using M\"obius transformation, we can arrange  $\mathbb{P}^1_T\simeq Y$  such that 
the composition $\vp:\mathbb{P}^1_T\simeq Y\stackrel{\alpha}{\lra} \PP^1_\la$ ramified at $T=0$ and $T=1$ over $\la=0$ 
with the ramification indices $3$ and $2$ respectively.  
Further, we may also assume that the three points over $\la=\infty$ are $T=a,b,\infty$ 
($a\neq b,\ a,b\in \C$) with 
ramification indices $2,2,1$ respectively. 
Thus, $\vp$ takes a form 
$$\vp(T)=c\frac{T^3(T-1)^2}{(T-a)^2(T-b)^2}=c\frac{T^3(T-1)^2}{(T^2-sT+p)^2},\ c,s=a+b,p=ab\in \C.$$
Observing the ramification over $\la=1$, we see that 
$$\vp(T)-1=c\frac{(T-u)^4(T-v)}{(T^2-sT+p)^2},\ u,v\in \C.$$
Solving this equation in $(c,p,s,u,v)$, we have 
$$(c,p,s,u,v)=\left(-\frac{16}{45},-\frac{1}{375},\frac{2}{15},-\frac{1}{5},-\frac{1}{80}\right),\ 
\left(\frac{4}{9},\frac{1}{3},\frac{4}{3},1,\frac{1}{4}\right).$$
The latter solution is discarded since it forces $a=b$. Thus, the former determines 
$\vp$ as desired. 
\end{proof}

\subsection{An algebraic Belyi function on $\wW$ via 
the intermediate cover $Y$}\label{algBelyi}
We now explicitly compute the isomorphism $\PP^1_T\simeq Y$ in Proposition \ref{BelyiY}. 
On the plane model $W$, the subgroup $H\simeq S_4$ is generated by permutations of
the coordinates together with sign changes. Put 
$$A=x^2+y^2+z^2,\  B=x^2y^2+y^2z^2+z^2x^2.$$
It is easy to see that the function
\begin{equation}
  t:=\frac{B}{A^2}
\end{equation}
is $H$-invariant and gives a coordinate on $Y=H\bs\wW\simeq\mathbb{P}^1_t$.
Therefore, the isomorphism 
$\PP^1_T\simeq Y\simeq \mathbb{P}^1_t$ is given by 
$$T=\ds\frac{at+b}{ct+d}$$ for some $a,b,c,d\in \C$ with $ad-bc\neq 0$. 
To determine $a,b,c,d$, we observe the ramification data of the following commutative diagram:
 \[
\xymatrix{
\PP^1_t \hspace{3mm} \ar[dr]_{\pi_{H,G}}^\circlearrowright 
\ar[r]^{\tiny t\mapsto T=\frac{at+b}{ct+d} } & \hspace{3mm}\PP^1_T \ar[d]^{\vp} \\
 & \hspace{3mm} \PP^1_\la.
}
\]
where $\pi_{H,G}:H\bs\wW\lra G\bs\wW$ is the natural projection defined by 
the inclusion $H\subset G$. 

Consider first the transposition
$\tau:[X:Y:Z]\longmapsto[Y:X:Z]$.  
Choose a fixed point $P_\tau=[1:1:\sqrt{-3}]$ of $\tau$ on $\wW$. 
Then, we have 
$$t(P_\tau)=\frac{x^2y^2+y^2z^2+z^2x^2}{(x^2+y^2+z^2)^2}=-5.$$
Notice that $\tau$ is conjugate to $\sigma_0^{3}=(45)$ but not to 
any power of $\sigma_x$ with $x\in \{1,\infty\}$. 
Thus, $\pi_{H,G}(t(P_\tau))=0$ (namely, $P_\tau\in \pi_G^{-1}(0)$). 
Further the ramification index of $\pi_{H,G}$ at $t=-5$ is 
$[\langle \sigma_0 \rangle:\langle \sigma^3_0 \rangle]=3$. 
On the other hand, by Proposition \ref{BelyiY}, the point $T=0$ on $\PP^1_T$ is a unique point 
above $\la=0$ with ramification index $3$.  
Since $T=0$ is the unique point above $\lambda=0$ with ramification index $3$,
we have 
$$\ds\frac{a(-5)+b}{c(-5)+d}=0.$$

Next, consider $\sigma:[X:Y:Z]\longmapsto[Y:Z:X]$.  
Choose a fixed point $P_\sigma=[1:1:1]$ of $\sigma$ on $W$. 
Then, we have 
$$t(P_\sigma)=\frac{x^2y^2+y^2z^2+z^2x^2}{(x^2+y^2+z^2)^2}=\frac{1}{3}$$
and it is a point of $\PP^1_t$ over $\lambda=0$ since $\sigma$ is conjugate to $\sigma^2_0$. 
Further the ramification index of $\pi_{H,G}$ at $t=\frac{1}{3}$ is $2$. 
The corresponding point is $T=1$ by Proposition \ref{BelyiY} again. Thus, we have 
$$\ds\frac{a\cdot \frac{1}{3}+b}{c\cdot \frac{1}{3}+d}=1.$$
Finally, 
by using $(12)(34)$ corresponding to the automorphism $[X:Y:Z]\mapsto [-X:Y:Z]$, 
we can check that $t=\infty$ corresponds to $T=-\frac{1}{80}$ as points over $\la=\infty$. 
Thus, we have $$\ds\frac{a}{c}=-\frac{1}{80}.$$
These three conditions determine the M\"obius transformation 
$$  T=-\frac{t+5}{16(5t-2)},$$
or equivalently  $t=\ds\frac{32T-5}{80T+1}$. 
Summing up, we have proved the following result with Proposition \ref{BelyiY}.
\begin{thm}\label{mainclaim} Recall the morphism $\vp$ in Proposition \ref{BelyiY}. 
The composition 
$$\wW\stackrel{[X:Y:Z]\mapsto t=B/A^2}{\lra}Y=H\bs \wW\stackrel{t\mapsto T=-\frac{t+5}{16(5t-2)}}{\lra} 
\PP_T\stackrel{\vp}{\lra}\PP^1_\la$$
yields the rational function 
$$\beta(X,Y,Z)=\frac{3125 \left(\frac{3 \left(X^2 Y^2+Y^2 Z^2+Z^2 X^2\right)}{\left(X^2+Y^2+Z^2\right)^2}-1\right)^2 \left(\frac{X^2 Y^2+Y^2 Z^2+Z^2 X^2}{\left(X^2+Y^2+Z^2\right)^2}+5\right)^3}{\left(\frac{5 \left(X^2 Y^2+Y^2 Z^2+Z^2 X^2\right)}{\left(X^2+Y^2+Z^2\right)^2}-2\right) \left(-\frac{75 \left(X^2 Y^2+Y^2 Z^2+Z^2 X^2\right)^2}{\left(X^2+Y^2+Z^2\right)^4}+\frac{1010 \left(X^2 Y^2+Y^2 Z^2+Z^2 X^2\right)}{\left(X^2+Y^2+Z^2\right)^2}+13\right)^2}$$
on $\wW$ 
is a generator of $\Q(\wW)^{S_5}=\Q(W)^{S_5}$ and it gives rise to an algebraic Belyi map over $\Q$:
$$\beta:\wW\lra S_5\bs \wW\simeq \PP^1_\la,\ [X:Y:Z]\mapsto \beta(X,Y,Z).$$
\end{thm}

\section{The modular parametrization and cusp forms for $\G_{\wW}$}\label{modularparacuspform}
Recall the classical modular lambda function
\begin{equation}\label{lambda}
  \lambda(\tau)=
  \frac{\vartheta_2(\tau)^4}{\vartheta_3(\tau)^4}=16q_2-128q^2_2+704q^3_2-3072q^4_2+\cdots,\ 
  q_2=e^{\pi i\tau}
\end{equation}
which yields an analytic isomorphism 
$$X(2)=\overline{\G(2)\bs \HH}\stackrel{\sim}{\lra} \PP^1_\la,\ \tau\mapsto \la(\tau).$$
See \cite[Section 1.2, 2.4]{DS} for the principal congruence subgroup $\G(2)$ of level 2 and 
the corresponding compactified modular curve $X(2)$. 
For three cusps $0,1,\infty$ of $X(2)$, 
$$\la(0)=1,\ \la(1)=\infty,\ \la(\infty)=0.$$
Recall the surjective homomorphism 
$\rho:\Gamma(2)\longrightarrow S_5$ associated with the Galois cover $\wW\to\mathbb{P}^1_\lambda$ (see Section \ref{intro}) and the finite index subgroup $\G_{\wW}:={\rm Ker}\rho$ of $\G(2)$.   
\begin{lem}\label{noncong} The group $\G_{\wW}$ is a non-congruence subgroup of $\SL_2(\Z)$ 
with index 720. 
\end{lem}
\begin{proof}Since $\overline{\G_{\wW}\bs \HH}\lra X(2)$ is a uniformization of 
$\wW\lra G\bs \wW\simeq \PP^1_\la$, 
the ramification indices of $\overline{\G_{\wW}\bs \HH}\lra X(2)$ 
over cusps $0,1,\infty$ are $4,2,6$ respectively.   
On the other hand, the ramification indices of the natural projection $X(2)\lra X(1)$ over $\infty$ 
is 2. 
Thus, the set of widths (or cusp amplitudes) of all cusps of $\G_{\wW}$ is given by 
$C:=\{8,4,12\}$. 
By \cite[Theorem 2]{Larcher}, we conclude that $\G_{\wW}$ is a non-congruence subgroup since 
$C$ is not closed under taking $\lcm$ 
(in fact, $\lcm(8,12)=24\not\in C$).

\end{proof}

\subsection{Puiseux expansions for $\vp(T)=\la(\tau)$}
Although the Puiseux expansions computed in this subsection will not be
used in the proof of the unbounded denominators conjecture, we include
them because they give explicit modular functions for
$\Gamma_{\wW}$. These expansions provide a complementary description
of the modular uniformization of the Wiman curve and may serve as a
starting point for studying further arithmetic properties of modular
functions for $\Gamma_{\wW}$.

We compute the Puiseux expansions of 
$$\vp(T)=\la(\tau)=16q_2-128q^2_2+704q^3_2-3072q^4_2+\cdots$$
around $\tau=\infty$, hence around $\la=\la(\infty)=0$. 
We refer to \cite[Section 8.3]{BK} for Puiseux expansions.  
Though the computation expects the unbounded denominators for  
modular functions with respect to $\G_{\wW}$, 
we leave the study of these functions for future work. 

The points $T=0$ and $T=1$ correspond to $\la=0$. 
We first compute the Puiseux expansions at $(T,\la)=(0,0)$. 
Put 
$$f(\la,T):=50000T^3(T-1)^2+\la(375T^2-50T-1)^2.$$
As for the Newton polygon of $f$ for $(\deg T,\deg \la)$, the first edge of the lower Newton polygon has slope $-1/3$,
and the remaining edge has slope $0$.

Thus, we can expand $T$ as 
$$T=\sum_{n=1}^\infty a_n \la^{\frac{n}{3}}\in \C[[\la^{\frac{1}{3}}]]$$
where 
$$\la^{\frac{1}{3}}=16^{\frac{1}{3}}q_6 (1-8q_2+\cdots)^{\frac{1}{3}}
=16^{\frac{1}{3}}q_6\left(1-\frac{8}{3}q_2+\frac{68}{9}q^2_2-\frac{1408}{81}q^3_2
+\frac{8702}{243}q^4_2+O(q^5_2)\right).$$
Here we choose the principal value for the cubic root.  
Determining the coefficients inductively, we have three Puiseux expansions 
\[
\begin{aligned}
T=T_a={}&
a q_6
+34a^2q_6^2
-\frac{651}{3125}q_6^3
-\frac{6134a}{3125}q_6^4 
+\frac{56659a^2}{3125}q_6^5
-\frac{3178824}{9765625}q_6^6
-\frac{182087014a}{29296875}q_6^7
+O(q_6^8),
\end{aligned}
\]
where  $q_6=e^{2\pi i \tau/6}$ and  $a^3=-\frac1{3125}$.

Similarly, around $(T,\la)=(1,0)$, we have two Puiseux expansions
\[
\begin{aligned}
T=T_b={}&
1
+bq_4
-\frac{69336}{3125}q_4^2
-\frac{36426b}{3125}q_4^3
+\frac{1762580736}{9765625}q_4^4 \\
&\quad
+\frac{738223836b}{9765625}q_4^5
-\frac{30218397241248}{30517578125}q_4^6
-\frac{2227793391032b}{6103515625}q_4^7 \\
&\quad
+\frac{411870280304719872}{95367431640625}q_4^8
+O(q_4^9),
\end{aligned}
\]
where  $q_4=e^{2\pi i \tau/4}$ and  $b^2=-\frac{104976}{3125}$. 

Thus, we have five modular functions $T_a$ and $T_b$ on $X(\G_{\wW})$. 
We leave the study of these modular functions, which appear to have unbounded denominators, for future work.

\subsection{$q_{12}$-expansions of certain local coordinates on the Wiman sextic}
In this subsection, we use another method to compute 
the $q_{12}$-expansions of $x=X/Z$ and  $y=Y/Z$.  

From Subsection \ref{algBelyi}, recall the fixed point
$
  P_\tau=[1:1:i\sqrt3]\in W
$ of $\tau$ lying over $t=-5$ and over $\la=0$.  
Regarding the projective coordinates of $W$, put 
$$
  x=\frac{X}{Z},\quad
  y=\frac{Y}{Z}.
$$
Then $x(P_\tau)=y(P_\tau)=-\frac{i}{\sqrt3}$. 
We introduce
\begin{equation}\label{mdxy}
  m=\frac{x+y}{2},\quad
  d=\frac{x-y}{2}, 
\end{equation}
and we have $m(P_\tau)=-\frac{i}{\sqrt3}$ and $d(P_\tau)=0$. 
Put $m_1=\left(-\frac{i}{\sqrt{3}}\right)^{-1}m-1$. 
Thus, solving 
\begin{eqnarray}\label{m1d}
f(m+d,m-d,1)&=&
f(\left(-\frac{i}{\sqrt{3}}\right)(m_1+1)+d,\left(-\frac{i}{\sqrt{3}}\right)(m_1+1)-d,1)\nonumber\\ 
&=&
\frac{2}{9} (-87 d^4 m_1^2+29 d^2 m_1^4+116 d^2 m_1^3+120 d^2 m_1^2-174 d^4 m_1+8 d^2 m_1+
27 d^6 \nonumber \\
&&-132 d^4-16 d^2-m_1^6
 -6 m_1^5-20 m_1^4-40 m_1^3-48 m_1^2-32 m_1)=0
\end{eqnarray}
 locally at $(m_1,d)=(0,0)$, we have 
\begin{equation}\label{md}
\left(-\frac{i}{\sqrt{3}}\right)^{-1}m=1
-\frac{1}{2}d^2
-\frac{37}{8}d^4
-\frac{55}{16}d^6
+\frac{967}{128}d^8
+\frac{24095}{256}d^{10}+\cdots\in \Z\Big[\frac{1}{2}\Big][[d]].
\end{equation}
Here, we can check all coefficients belong to $\Z[\frac{1}{2}]$ by using 
Hensel's lemma (cf. \cite[Chapter VI.1, Lemma 1.2.]{Sil}). 
The $S_4$-invariant coordinate $t$ has the expansion
\begin{eqnarray}\label{td}
  t&=&B/A^2=\frac{x^2 y^2+x^2+y^2}{(x^2+y^2+1)^2}
  =\frac{(m+d)^2 (m-d)^2+(m+d)^2+(m-d)^2}{((m+d)^2+(m-d)^2+1)^2} \nonumber \\
  &=& \frac{(m^2-d^2)^2+2m^2+2d^2}{(2m^2+2d^2+1)^2} \nonumber \\
  &=&
  -5
+108d^2
-1188d^4
+9666d^6
-\frac{133677}{2}d^8
+\frac{3357261}{8}d^{10}
+\cdots \in \Z\left[\frac{1}{2}\right][[d]].
\end{eqnarray}
Using the relation between $t$ and $\lambda$ via $\vp$ with (\ref{td}), one obtains
\begin{eqnarray}\label{lamd}
 \lambda&=&\vp\left(\frac{-t-5}{16 (5 t-2)}\right)
  =\frac{3125 (t+5)^3 (3 t-1)^2}{(5 t-2) \left(75 t^2-1010 t-13\right)^2} \nonumber \\
&=&  5^5d^6\left(-\frac{1}{4}
-\frac{3}{8}d^2
-\frac{27}{64}d^{4}
-\frac{3537}{32}d^{6}
-\frac{330885}{1024}d^{8}+\cdots\right)\in 5^5 d^6 \Z\left[\frac{1}{2}\right][[d^2]].
\end{eqnarray}

Let $\kappa$ satisfy $\kappa^6=-\ds\frac{64}{3125}$.
Then, expanding $d$ in (\ref{lamd}) in terms of $q_{12}=e^{2\pi i \tau/12}$, we have  
\begin{equation}
  d(q_{12})
  =
  \kappa q_{12}
  -\frac{\kappa^3}{4}q_{12}^3
  +\frac{\kappa^5}{16}q_{12}^5
  +O(q_{12}^7) \in \Z\Big[\frac{1}{2},\kappa\Big][[q_{12}]].
\end{equation}
More precisely, we can also write $d(q_{12})$ as 
\begin{equation}\label{dq12}
d(q_{12})=\kappa q_{12}\sum_{n\ge 0}f_n  \kappa^{2n} q_{12}^{2n},\ f_n \in 
\Z\Big[\frac{1}{2}\Big],\ f_0=1.
\end{equation}
Substituting $d(q_{12})$ into (\ref{md}) and (\ref{mdxy}), we have 
\begin{align}
  x(q_{12})
  ={}&
  -\frac{i}{\sqrt3}
  +\kappa q_{12}
  +\frac{\sqrt3 i}{6}\kappa^2 q_{12}^2
  -\frac14\kappa^3 q_{12}^3
  +\frac{35\sqrt3 i}{24}\kappa^4 q_{12}^4
  + O(q_{12}^5)\in \O_K\Big[\frac{1}{2}\Big][[\kappa q_{12}]], \nonumber \\
  y(q_{12})
  ={}&
  -\frac{i}{\sqrt3}
  -\kappa q_{12}
  +\frac{\sqrt3 i}{6}\kappa^2 q_{12}^2
  +\frac14\kappa^3 q_{12}^3
  +\frac{35\sqrt3 i}{24}\kappa^4 q_{12}^4
  +O(q_{12}^5)\in \O_K\Big[\frac{1}{2}\Big][[\kappa q_{12}]], \nonumber 
\end{align}
where $K=\Q(\zeta_6)=\Q(\sqrt{-3})$.

\subsection{Holomorphic differentials}\label{holdiffe}
Let $F(x,y):=f(x,y,1)$.
Recall that $H^0(\wW,\Omega^1_{\wW})$ has 
a basis given by 
$$
  \omega_1
  =
  \frac{x(x^2-1)\,dx}{F_y(x,y)},\ 
  \omega_2
  =
  \frac{y(x^2-1)\,dx}{F_y(x,y)},\ 
  \omega_3
  =
  \frac{(x^2-1)\,dx}{F_y(x,y)},
$$
$$
  \omega_4
  =
  \frac{x(y^2-1)\,dx}{F_y(x,y)},\ 
  \omega_5
  =
  \frac{y(y^2-1)\,dx}{F_y(x,y)},\ 
  \omega_6
  =
  \frac{(y^2-1)\,dx}{F_y(x,y)}.
$$

Substituting the expansions $x=x(q_{12})$ and $y=y(q_{12})$, we write
\[
  \omega_j=A_j(q_{12})\,dq_{12}=\frac{\pi i}{6}q_{12} A_j(q_{12})\,d\tau\quad (1\le j\le 6).
\]
where 
\begin{align}
  A_1(q_{12})
  ={}&
  -\frac{\kappa}{8}
  -\frac{3\sqrt3 i}{16}\kappa^2 q_{12}
  +\frac14\kappa^3 q_{12}^2
  +\frac{3\sqrt3 i}{64}\kappa^4 q_{12}^3
  +\frac{35}{128}\kappa^5 q_{12}^4
  +\cdots,  \nonumber\\
  A_2(q_{12})
  ={}&
  -\frac{\kappa}{8}
  +\frac{\sqrt3 i}{16}\kappa^2 q_{12}
  -\frac18\kappa^3q_{12}^2
  -\frac{\sqrt3 i}{64}\kappa^4q_{12}^3
  +\frac{35}{128}\kappa^5q_{12}^4
  +\cdots,\nonumber\\
  A_3(q_{12})
  ={}&
  -\frac{\sqrt3 i}{8}\kappa
  +\frac{3}{16}\kappa^2 q_{12}
  +\frac{3}{64}\kappa^4q_{12}^3
  -\frac{35\sqrt3 i}{128}\kappa^5q_{12}^4
  +\cdots,\nonumber\\
  A_4(q_{12})
  ={}&
  -\frac{\kappa}{8}
  -\frac{\sqrt3 i}{16}\kappa^2 q_{12}
  -\frac18\kappa^3q_{12}^2
  +\frac{\sqrt3 i}{64}\kappa^4q_{12}^3
  +\frac{35}{128}\kappa^5q_{12}^4
  +\cdots,\nonumber\\
  A_5(q_{12})
  ={}&
  -\frac{\kappa}{8}
  +\frac{3\sqrt3 i}{16}\kappa^2 q_{12}
  +\frac14\kappa^3q_{12}^2
  -\frac{3\sqrt3 i}{64}\kappa^4q_{12}^3
  +\frac{35}{128}\kappa^5q_{12}^4
  +\cdots,\nonumber\\
  A_6(q_{12})
  ={}&
  -\frac{\sqrt3 i}{8}\kappa
  -\frac{3}{16}\kappa^2 q_{12}
  -\frac{3}{64}\kappa^4q_{12}^3
  -\frac{35\sqrt3 i}{128}\kappa^5q_{12}^4
  +\cdots, \nonumber
\end{align}
and they all belong to $\kappa \O_K\Big[\ds\frac{1}{2}\Big][[\kappa q_{12}]]$. 
Here, $\kappa$ satisfies $\kappa^6=-\frac{64}{3125}=-2^6\cdot 5^{-5}$ and $K=\Q(\zeta_6)=\Q(\sqrt{-3})$. 
Then, 
\begin{equation}
  g_j(\tau):=\frac{6}{\pi i}\frac{\omega_j}{d\tau}
  =q_{12}A_j(q_{12})\ (1\le j\le 6) 
\end{equation}
make up a basis of $S_2(\G_{\wW})\simeq H^0(X_{\G_{\wW}},\Omega^1_{X_{\G_{\wW}}})$. 
Summing up, we have proved the following result.
\begin{prop}\label{FCalg}
Keep the notation as above. Then $S_2(\G_{\wW})$ admits a basis whose 
Fourier expansions in $q_{12}$ belong to $\O_K\Big[\ds\frac{1}{2}\Big][[\kappa q_{12}]]$. 
\end{prop}

\section{The unbounded denominators conjecture for $S_2(\G_{\wW})$}\label{unbounded-dc}
In this section, we prove the following claim:
\begin{thm}\label{udc} 
Every nonzero cusp form in $S_2(\G_{\wW})$ whose Fourier expansion 
in $q_{12}$ has coefficients in $\bQ$ has unbounded denominators.
\end{thm}

\subsection{Preliminaries}\label{pre} 
Recall $\omega_1,\ldots,\omega_6$ from Subsection \ref{holdiffe}. 
Put
\[
 V_{\Q}:=H^0(\wW,\Omega_{\wW}^1)
 =\bigoplus_{j=1}^6\Q\omega_j,
 \quad
 V_{\overline{\Q}}
 :=V_{\Q}\otimes_{\Q}\overline{\Q}.
\]

Let \(\sigma_0=(132)(45)\) which acts on $W$ on the left.   
A direct computation gives
\[
  \sigma_0=(14)\alpha(142) 
\]
where the two elements of \(S_4\) occurring here act by
\[
  h_1=(14):[X:Y:Z]\mapsto [Y:X:-Z],
  \quad
  h_2=(142):[X:Y:Z]\mapsto [Y:Z:-X]
\]
(see Section \ref{intro} for the description of $\Aut(\wW)=\Bir(W)\simeq S_5$). 

Write
\[
  \boldsymbol\omega
  ={}^t (\omega_1,\ldots,\omega_6),
  \quad
  \sigma^*\boldsymbol\omega=M_\sigma\boldsymbol\omega,\ \sigma\in S_5 
\]
By direct computation, we have 
\[
{\small
  M_{h_1}
  =
  \begin{pmatrix}
    0&0&0&0&1&0\\
    0&0&0&1&0&0\\
    0&0&0&0&0&-1\\
    0&1&0&0&0&0\\
    1&0&0&0&0&0\\
    0&0&-1&0&0&0
  \end{pmatrix},\ 
  M_\alpha
  =
  \frac12
  \begin{pmatrix}
  -1&1&0&-1&0&-1\\
  -1&1&0&1&0&1\\
  1&1&2&-1&0&-1\\
  -1&0&1&0&-1&-1\\
  -1&0&1&2&1&-1\\
  1&2&1&0&-1&-1
  \end{pmatrix},\ 
  M_{h_2}
  =
  \begin{pmatrix}
    0&1&0&0&-1&0\\
    0&0&1&0&0&-1\\
    -1&0&0&1&0&0\\
    0&1&0&0&0&0\\
    0&0&1&0&0&0\\
    -1&0&0&0&0&0
  \end{pmatrix}.
  }
\]
Thus, we have 
\[
  M_{\sigma_0}=M_{h_1}M_\alpha M_{h_2}
  =
  \frac12
  \begin{pmatrix}
  0&1&1&1&1&0\\
  0&-1&-1&1&1&0\\
  0&-1&-1&-1&1&2\\
  -1&0&1&0&1&-1\\
  1&-2&1&0&1&-1\\
  1&0&-1&-2&1&1
  \end{pmatrix}.
\]
In particular, $M_{\sigma_0}^6=I_6$ and 
$ \det(YI_6-M_{\sigma_0})=Y^6-1$.

\begin{lem}\label{cyclic} Keep the notation as above. 
Then, $\omega_2$ is a cyclic vector of $V_{\bQ}$ for $M_{\sigma_0}$. 
Namely, $V_{\overline{\Q}}=\overline{\Q}[\sigma^*_0]\omega_2$.
\end{lem}
\begin{proof}
The claim follows from 
\[
 \det\bigl(
 \omega_2,M_{\sigma_0}\omega_2,\ldots,M_{\sigma_0}^5\omega_2
 \bigr)=\frac74\neq0.
\]
\end{proof}

\subsection{A rational parametrizaion of the Wiman curve over $\F_5$}\label{rat}
As explained in \cite[Remark 3.1]{GKLM}, $W_{\F_5}$ is a rational curve.  
For later use, we need to describe a rational parametrization of $W_{\F_5}$. 
We verify the rationality of \(W_{\F_5}\) and compute its
parametrization using Magma (\cite{Bosma}). 

Let $z$ be a (local) parameter of $\PP^1_{\F_5}$.  
Then, 
\begin{equation}\label{xypara}
 x=\frac{(z^2+1)^3}{z^3},
 \quad
 y=\frac{2(z^2-1)^3}{z^3}.
\end{equation}
gives a birational map $\PP^1_{\F_5}\lra W_{\F_5}$
with the converse 
$$\frac{x^4 - 2 x^3 y - x^2 y^2 + 2 x y^3 + y^4 - 1}{x^5 + x^4 y - 2 x^3 y^2 - 2 x^3 - x^2 y^3 + 2 x^2 y + 2 x y^4 - 
 x y^2 + 2 x - 2 y^5 - y^3 + y}.$$
Thus, we have 
$\F_5(W)=\F_5(z)$. 

Here is a Magma code:
\lstdefinelanguage{Magma}{
  morekeywords={
    GF,ProjectiveSpace,Curve,IsIrreducible,Genus,
    IsSingular,IsAnalyticallyIrreducible,Places,Degree,
    Parametrization,DefiningIdeal,Image,
    DefiningEquations
  },
  sensitive=true,
  morecomment=[l]{//},
  morestring=[b]"
}

\lstset{
  basicstyle=\ttfamily\small,
  keywordstyle=\color{blue},
  commentstyle=\color{gray},
  numbers=left,
  numberstyle=\tiny,
  stepnumber=1,
  numbersep=8pt,
  breaklines=true,
  columns=fullflexible,
  frame=single,
  showstringspaces=false
}

\begin{lstlisting}[language=Magma,
 caption={Computation of a rational parametrization of \(W_{\F_5}\).},
  label={code:wiman-parametrization},
  captionpos=t,
  belowcaptionskip=12pt]
k := GF(5);

P2<X,Y,Z> := ProjectiveSpace(k,2);

f :=    X^6 + Y^6 + Z^6  + (X^2 + Y^2 + Z^2)*(X^4 + Y^4 + Z^4)  - 12*X^2*Y^2*Z^2;

C := Curve(P2,f);

IsIrreducible(C);
Genus(C);

p := C![1,2,0];

IsSingular(C,p);
IsAnalyticallyIrreducible(C,p);

pls := Places(p);
[ Degree(pl) : pl in pls ];

pl := pls[1];

P1amb<s,t> := ProjectiveSpace(k,1);
P1 := Curve(P1amb);

prm := Parametrization(C,pl,P1);
prm;

DefiningIdeal(Image(prm)) eq DefiningIdeal(C);
\end{lstlisting}

\subsection{Expansion of the cycles}
Fix an embedding
$
 \iota_5:\overline{\Q}\hookrightarrow\overline{\Q}_5
$
compatible with the choices of a primitive sixth root of unity $\zeta_6$
and a sixth root $\kappa$ satisfying $\kappa^6=-\frac{64}{3125}$. Recall $K=\Q(\zeta_6)$. 
The embedding $\iota_5$ together with the reduction map induces $\O_K[\frac{1}{2}]\hookrightarrow \bZ_5\rightarrow \bF_5$. 
We denote it by $\overline{\iota}_5$.  

Put 
\[
 \omega_2=J_2(Q)\frac{dQ}{Q},
 \quad
 J_2(Q)=\sum_{n\geq0}a_nQ^n \in \O_K[\frac{1}{2}][[Q]],\ Q:=\kappa q_{12}
\]
and for $0\le i\le 5$, 
\[
 \eta_i
 :=\frac16\sum_{k=0}^5
 \zeta_6^{-ik}(\sigma^*_0)^k\omega_2=\left(\sum_{m\geq0}a_{6m+i}Q^{6m+i}\right)\frac{dQ}{Q}
\]

We can rewrite (\ref{lamd}) as 
$$\la=-\frac{5^5}{4}d^6\Phi(d^2)$$ with 
$\Phi\in 1+T'\Z[\frac{1}{2}][[T']]$.
Recall 
$t=\ds\frac{x^2 y^2+x^2+y^2}{(x^2+y^2+1)^2}$ (see (\ref{td})). 
Using $\vp$ and (\ref{xypara}), we have 
\begin{equation}\label{eq1}
d^6\Phi(d^2)=-\frac{4}{5^5}\la=-\frac{4}{5^5}\vp\Big(\frac{-t-5}{16(5t-2)}\Big)
\equiv 
 \frac{2 p(z)^6}
 {z^{20}(z-2)^{20}(z-1)^{20}(z+1)^{20}(z+2)^{20}} \quad ({\rm mod}\ 5)
\end{equation}
where 
$$p(z)=\left(z^2+2\right) \left(z^2+3\right) \left(z^2+z+1\right) \left(z^2+z+2\right) 
\left(z^2+2 z+3\right) \left(z^2+2 z+4\right)\times $$
$$ \left(z^2+3 z+3\right) \left(z^2+3 z+4\right) \left(z^2+4 z+1\right) \left(z^2+4 z+2\right)$$
is a separable polynomial over $\F_5$ of $\deg(p)=20$. 

Since $q_2=q_{12}^6=\frac{Q^6}{\kappa^6}=-\frac{3125}{64}Q^6$, if we write 
$\la(q_2)=16 q_2 L(q_2)$ with $L(q_2)\in 1+q_2\Z[[q_2]]$, then we have 
$$-\frac{5^5}{4}d^6\Phi(d^2)=\la=16 q_2 L(q_2)=-\frac{3125}{4}Q^6 L(-\frac{3125}{4}Q^6).$$
Thus, $d^6\Phi(d^2)=Q^6 L(-\frac{3125}{4}Q^6)\equiv Q^6$ mod $5$ as elements of 
$\F_5[[d,Q]]$. 
Plugging this into (\ref{eq1}), we have 
\begin{equation}\label{Tz}
Q^6\equiv  \frac{2 p(z)^6}
 {z^{20}(z-2)^{20}(z-1)^{20}(z+1)^{20}(z+2)^{20}}=:P(z) \quad ({\rm mod}\ 5).
 \end{equation} 
Using (\ref{xypara}), 
we write
\[
 \eta_i\equiv M_i(z)\,dz\quad ({\rm mod}\ 5), \quad (0\le i\le 5).
\]
Explicitly, if we put $\zeta=\overline{\iota}_5(e^{2\pi i/6})$, then 
$$M_i(z)=\frac{N_i(z)}{D(z)},$$
where
\[
  D(z)
  =
  z^2(z-1)^2(z+1)^2(z-2)^2(z+2)^2
\]
and
\[{\small
\begin{aligned}
N_0(z)
={}&
-2z^{10}-z^8+z^6-z^5+2z^4+2z^2+1,
\\
N_1(z)
={}&
\zeta\bigl(
z^{10}+2z^8-2z^7+z^6-z^5+z^4
-2z^3+2z^2-2
\bigr)
+z^{10}+z^8-2z^7+2z^6+z^5
-2z^4-2z^3-z^2-2,
\\
N_2(z)
={}&
\zeta\bigl(z^{10}+2z^7+2z^3-2\bigr)
-z^7-z^3+2,
\\
N_3(z)
={}&
-z^{10}+z^7+z^3+2,
\\
N_4(z)
={}&
\zeta\bigl(-z^{10}-2z^7-2z^3+2\bigr)
+z^{10}+z^7+z^3,
\\
N_5(z)
={}&
\zeta\bigl(
-z^{10}-2z^8+2z^7-z^6+z^5-z^4
+2z^3-2z^2+2
\bigr)
+2z^{10}-2z^8+z^7-2z^6-z^4+z^3+z^2+1.
\end{aligned}
}
\]
Choose a root $\xi\in \F=\F_{25}$ of $p(z)$ and solving (\ref{Tz}) implicitly 
around $(Q,z)=(0,\xi)$, we can formally write $z=z(Q)=\xi+c_1Q+\cdots\in \F[[Q]]$. 
Taking the logarithmic derivative of (\ref{Tz}) in $Q$, 
we have 
\[
 \frac{dQ}{Q}=z'(Q)\frac{P'(z(Q))}{P(z(Q))}dQ
\]
Consequently,
$$ \eta_i\equiv M_i(z)\,dz=M_i(z(Q))z'(Q)dQ=R_i(z)\frac{dQ}{Q}\quad ({\rm mod}\ 5),\ 
R_i(z):=\frac{P(z)M_i(z)}{P'(z)},\quad (0\le i\le 5). $$
Thus,
\[
 \sum_{m\geq0}\overline a_{6m+i}Q^{6m+i}=R_i(z(Q)),\ \overline a_{6m+i}:=
 \overline{\iota}_5(a_{6m+i}).
\]
We note that $\deg_z(R_i)=30$ 
as a rational function in $z$. 

\subsection{A proof of the main result}
Suppose that, for some
$i$ ($0\le i\le 5$), only finitely many coefficients $\overline a_{6m+i}$ are nonzero.
By Lemme \ref{cyclic}, there exists $G_i(Y)\in\F_{25}[Y]$ such that
\[
 R_i(z(Q))=Q^{i} G_i(Q^6).
\]
Taking sixth powers and using (\ref{Tz}) $Q^6=P(z)$ over $\F$ gives 
\begin{equation}\label{rpf}
 R_i(z)^6=P(z)^{i} G_i(P(z))^6.
\end{equation}

We explain that no cancellation occurs in the degree computation on the right 
hand side of (\ref{rpf}). 
It is easy to see that $ \deg_z P=120$ where $\deg_z$ means the degree as a rational 
function in $z$. 
Put $F_i(Y)=Y^i G_i(Y)^6$. Then
\[
  \deg F_i=i+6\deg G_i.
\]
Since degrees of nonconstant rational maps in $z$ multiply under composition,
we have
\[
  \deg_z F_i(P(z))
  =
  \deg F_i\cdot\deg_z P
  =
  120\bigl(i+6\deg G_i\bigr).
\]

Comparing degrees of both sides of (\ref{rpf}) , we obtain
\[
  180
  =
  120\bigl(i+6\deg G_i\bigr),
\]
or equivalently,
\[
  3=2i+12\deg G_i.
\]
But this is impossible.

We conclude that there is no polynomial
\(G_i(Y)\in\F_{25}[Y]\) such that
\[
  R_i(z(Q))=Q^i G_i(Q^6).
\]

Therefore, 
\begin{equation}\label{non-poly}
 \sum_{m\geq0}\overline a_{6m+i}Q^{6m+i}
 \notin\F_{25}[Q]
 \quad (i=0,\ldots,5).
\end{equation}
Thus, for every $i$, there are infinitely many $m$ such that the normalized
coefficient in degree $6m+i$ is a $5$-adic unit.  
Summing up, we have proved the following result: 
\begin{thm}\label{premain} Keep the notation as above. 
For each $0\le i\le 5$, the expansion of $\eta_i$ at $q_{12}$ has (5-adic)
unbounded denominators. In particular, it is true for $\omega_2=\eta_0$. 
\end{thm}
\begin{proof} Write $\eta_i=
\ds\sum_{m\ge 0}a_{6m+i}\kappa^{6m+i}q_{12}^{6m+i}\frac{dq_{12}}{q_{12}}\in 
\O_K\Big[\frac{1}{2}\Big][[q_{12}]]\frac{dq_{12}}{q_{12}}$. 
By (\ref{non-poly}), there are infinitely many $m$ such that $a_{6m+i}$ is 
5-adically unit. Thus, $\ds\inf_m\{\ord_5(a_{6m+i}\kappa^{6m+i})\}=-\infty$. 
Here $\ord_5$ is normalized to be $\ord_5(5)=1$ so that $\ord_5(\kappa)=-\frac{5}{6}$. 
\end{proof}
We are now ready to prove the second main result:
\begin{thm}\label{diffubc}
 Every nonzero differential $\omega\in V_{\overline{\Q}}$ 
 has unbounded $5$-adic denominators in its Fourier expansion in $q_{12}$, with respect to the fixed embedding $\iota_5$.
\end{thm}

\begin{proof}
Since $\omega_2$ is a cyclic vector for $\sigma^*$, every nonzero
$\omega\in V_{\overline{\Q}}$ can be written uniquely as
\[
 \omega=G(\sigma^*_0)\omega_2,
 \quad
 G(Y)\in\overline{\Q}[Y],
 \quad
 \deg_Y G<6.
\]
Choose a number field $L$ containing the coefficients of $G$, $\zeta_6$,
and $\kappa$, and let $\mathfrak p$ be the place of $L$ induced by
$\iota_5$.  Let $\O_{L,\mathfrak p}$ be the localization of $\O_L$ at $\frak p$. 

Multiplication of $G$ by a nonzero scalar changes all valuations
by only a fixed constant and hence does not affect unboundedness.  We may
therefore normalize $G$ so that
$
 G(Y)\in\mathcal O_{L,\mathfrak p}[Y]
$
and its reduction
$
 \overline G(Y)\in\overline{\F}_5[Y]
$
is nonzero.  We still have $\deg_Y\overline G<6$.

The polynomial $Y^6-1$ has six distinct roots in $\overline{\F}_5$, since
$5\nmid6$.  Hence $\overline G$ cannot vanish at all of them and there exists
$i\in\{0,\ldots,5\}$ such that
\[
 \overline G(\zeta_6^i)\neq0.
\]
Thus $G(\zeta_6^i)$ is a $\mathfrak p$-adic unit.  
We write $\omega=G(\sigma^*_0)\omega_2=\ds\sum_{n\ge 1}b_n Q^n\frac{dQ}{Q} \in 
\O_{L,\frak p}[[q_{12}]]\frac{dQ}{Q}$ with $Q=\kappa q_{12}$. 
Then, $b_{6m+i}=G(\zeta_6^i) a_{6m+i}$ for any $m\ge 1$ and the above chosen $i$. 
By Theorem \ref{premain}, 
\[
 \inf_m\ord_{\mathfrak p}(a_{6m+i}\kappa^{6m+i})=-\infty,
\]
and thus, $\ds\inf_m\ord_{\mathfrak p}(b_{6m+i} \kappa^{6m+i})=-\infty$. 
Therefore, the Fourier expansion of $\omega$ in $q_{12}$ has unbounded $5$-adic denominators.
\end{proof}

\begin{proof}(A proof of Theorem \ref{udc}.) 
Let $S_2(\G_{\wW},\bQ)$ denote the space of cusp forms whose Fourier coefficients in their $q_{12}$-expansions all lie in $\bQ$. We have a $\bQ$-linear isomorphism 
$S_2(\G_{\wW},\bQ)\simeq V_{\bQ}$ by sending $f$ to 
$\frac{\pi i}{6}f(\tau)d\tau=q^{-1}_{12}f(q_{12})dq_{12}$. 
The claim now follows from Theorem \ref{diffubc}. 
\end{proof}

\section{Appendix} 
In this appendix, we explicitly compute a generator of $\Q(\wW)^{A_5}$. 
Observing the ramification of the $S_5$-covering $\wW\lra S_5\bs \wW$, as in Subsection \ref{algBelyi}, we can see easily that 
the intermediate Galois covering  
$\wW\lra A_5\bs\wW$ is ramified over four points  with ramifications indices 
$3,2,2,2$ respectively. By the Hurwitz genus formula, 
$$2g(\wW)-2=|A_5|(2g(A_5\bs\wW)-2)+60\left(1-\frac{1}{3}\right)
+60\left(1-\frac{1}{2}\right)+60\left(1-\frac{1}{2}\right)+60\left(1-\frac{1}{2}\right).$$
Thus, we have $g(A_5\bs\wW)=0$. 
We shall compute a generator of $\Q(\wW)^{A_5}$  which gives a local parameter $s$ of $\PP^1_s$. 

A key is to observe the equation arising from (\ref{laT}):
$$f(\la,T):=50000T^3(T-1)^2+\la (375T^2-50T-1)^2=0.$$
The decomposition field of $f(\la,T)$ over $\Q(\la)$ is $\Q(\wW)$. 
Let $T_1,\ldots,T_5\in \Q(\wW)$ be the roots of $f(\la,T)=0$ and 
the $S_5$-action permutes them. 
Put 
$D_\la:=\ds\prod_{1\le i<j\le 5}(T_i-T_j)$. 
The discriminant of $f(\la,T)$ in $T$ is given by 
$$D_\la^2=2^{26}\cdot 3^{11}\cdot 5^{25}\la^3(1-\la)^3.$$
We now consider 
$$y:=\frac{D_\la}{2^{13}\cdot 3^{5}\cdot 5^{12}\la(1-\la)}$$
so that we have 
$$y^2=15\la (1-\la).$$
Obviously, $\Q(\wW)^{A_5}\supset \Q(\la,y)\supset \Q(\la)$. Since 
$[\Q(\wW)^{A_5}:\Q(\la)]=[S_5:A_5]=2$ and 
$\Q(\la,y)\neq \Q(\la)$ ($y$ is not $S_5$-invariant). Thus, 
$\Q(\wW)^{A_5}=\Q(\la,y)$. 
Put $s=\frac{y}{\la}$. Then, we can solve the above quadratic equation as 
$$(\la,y)=\left(\frac{15}{s^2+15},\frac{15s}{s^2+15}\right)$$
so that $\Q(\wW)^{A_5}=\Q(\la,y)=\Q(s)$.
We have only to compute 
$$s=\frac{y}{\la}=\frac{\sqrt{15\la(1-\la)}}{\la}$$ 
in terms of the coordinates of the Wiman curve.  
Substituting $\la=\beta(X,Y,Z)$ in Theorem \ref{mainclaim} into $15\la(1-\la)$, 
we have 
$$15\la(1-\la)=2^{-6}\cdot 3^8\cdot 5^6 \cdot\frac{S_1^2 (19 S_1^2 - 25 S_2)^4 
(X^2 - Y^2)^6 (Y^2 - Z^2)^6 (Z^2 - X^2)^6 }{(3 S_2 - S_1^2)^4 (2 S_1^2 - 5 S_2)^2 
(13 S_1^4 + 1010 S_1^2 S_2 - 75 S_2^2)^4}$$
where 
$$S_1:=X^2+Y^2+Z^2,\ S_2:=X^2 Y^2 + Y^2 Z^2 + Z^2 X^2.$$ 
Therefore, (freely choosing a sign of the square roots) 
$$s=-
2^{-3}\cdot 3^4\cdot 5^3 \cdot\frac{S_1  (19 S_1^2 - 25 S_2)^2 
(X^2 - Y^2)^3 (Y^2 - Z^2)^3 (Z^2 - X^2)^3 }{(3 S_2 - S_1^2)^2 (2 S_1^2 - 5 S_2) 
(13 S_1^4 + 1010 S_1^2 S_2 - 75 S_2^2)^2 }\times \frac{1}{\beta(X,Y,Z)}$$
{\small
$$
=\frac{81 (X^2 - Y^2)^3  (Y^2 - Z^2)^3 ( Z^2-X^2)^3 (X^2 +
      Y^2 + Z^2) \{19 (X^4 + Y^4 +  Z^4)+ 13( X^2 Y^2 +Y^2 Z^2+ Z^2 X^2   )\}^2}
      {200 (X^4 + Y^4+ Z^4- X^2 Y^2  - 
     Y^2 Z^2- Z^2 X^2 )^4 \{5( X^4 + Y^4+  Z^4)+ 11 (X^2 Y^2 + 
     Y^2 Z^2 +  Z^2 X^2  )\}^3}$$
     }
 is an explicit generator of $\Q(\wW)^{A_5}$ as desired.

\section*{Declaration of generative AI assistance}

The research questions, conceptual framework, and principal mathematical ideas presented in this paper were developed by the authors. During the preparation of the manuscript, they used ChatGPT, developed by OpenAI, to improve the language and presentation and to assist with explicit algebraic computations. 
There is no further use of AI in this paper. 
All AI-assisted material was independently examined and verified by the authors, who take full responsibility for the correctness and content of the manuscript.

\end{document}